\documentclass[]{amsart}

\usepackage{amsmath}
\usepackage{amsthm}
\usepackage{amsfonts}
\usepackage{amssymb}
\usepackage{enumitem}
\usepackage{todonotes}
\usepackage{mathtools}
\usepackage{mathrsfs,enumitem,upgreek,mathtools}

\usepackage{hyperref}
\hypersetup{colorlinks=true}

\newcommand{\CH}{\mathsf{CH}}		
\newcommand{\forcing}[1]{\mathbf{ #1}} 	
\newcommand{\FORALL}[1]{\forall {#1} \, }
\newcommand{\EXISTS}[1]{\exists {#1} \, }
\newcommand{\Q}{\mathbb{Q}}	
\newcommand{\R}{\mathbb{R}}	
\newcommand{\ZF}{\mathsf{ZF}} 		
\newcommand{\ZFC}{\mathsf{ZFC}}		
\newcommand{\Pow}{\mathscr{P}}		
\newcommand{\IMPLIES}{\Rightarrow }
\newcommand{\IFF}{\Leftrightarrow}

\newcommand{\markdef}[1]{\emph{#1}} 
\newcommand{\Mid}{\boldsymbol\mid}			
\newcommand{\setof}[2]{\mathopen \{{#1}\Mid{#2} \mathclose\}} 
\newcommand{\setofLR}[2]{\left \{{#1} \Mid {#2} \right\}} 
\newcommand{\set}[1]{\mathopen \{ {#1} \mathclose \}} 
\newcommand{\seqof}[2]{\mathopen \langle #1 \Mid #2 \mathclose \rangle} 
\newcommand{\Vv}{\mathord{\mathrm{V}}}		
\newcommand{\Ll}{\mathord{\mathrm{L}}}		
\newcommand{\Hh}{\mathord{\mathrm{H}}}		

\newcommand{\bDelta}{\boldsymbol{\Delta}}
\newcommand{\bGamma}{\boldsymbol{\Gamma}}

\DeclareMathOperator{\TC}{TC} 
\newcommand{\Models}{\vDash}
\newcommand{\card}[1]{\mathopen{ |} #1 \mathclose{ |}}
\newcommand{\leqc}{\leq_{\mathrm{c}}}
\newcommand{\eqc}{=_{\mathrm{c}}}
\newcommand{\eq}[1]{\boldsymbol{[} #1 \boldsymbol{]}}
\newcommand{\constructibledegrees }{ \boldsymbol{\mathscr{D}}_{\mathrm{c}}}
\renewcommand{\restriction}{\mathop{\upharpoonright}}

\theoremstyle{plain}
\newtheorem{theorem}{Theorem}[]

\newtheorem{proposition}[theorem]{Proposition}
\newtheorem{lemma}[theorem]{Lemma}
\newtheorem{corollary}[theorem]{Corollary}

\theoremstyle{definition}
\newtheorem{definition}[theorem]{Definition}

\newtheorem{question}[theorem]{Question}
\newtheorem*{question*}{Question}

\theoremstyle{remark}
\newtheorem{remark}[theorem]{Remark}

\newtheorem*{remarks*}{Remarks}

\usepackage[authormarkup=none, authormarkuptext=name, xcolor=dvipsnames]{changes}
\definechangesauthor[name={Alessandro}, color={blue}]{0}
\definechangesauthor[name={Lorenzo}, color={blue}]{1}

\usepackage[style=alphabetic,sorting=nyt,backend=bibtex8]{biblatex}
\bibliography{DefSierp.bib}

\usepackage{hyperref}

\begin{document}

\title[Constructibility degrees and Sierpi\'{n}ski's coverings]{The breadth of constructibility degrees and definable Sierpi\'{n}ski's coverings}
\author{Alessandro Andretta}
\address{Università degli Studi di Torino, Dipartimento di Matematica ``G. Peano", Via Carlo Alberto 10, 10123 Torino, Italy}
\curraddr{}
\email{alessandro.andretta@unito.it}
\author{Lorenzo Notaro}
\address{University of Vienna, Institute of Mathematics, Kurt G\"{o}del Research Center, Kolingasse 14-16, 1090 Vienna, Austria}
\curraddr{}
\email{lorenzo.notaro@univie.ac.at}

\thanks{This research was supported by the project PRIN 2022 “Models, sets and classifications”, prot. 2022TECZJA. 
The second author would also like to acknowledge INdAM for the financial support. The research of the second author was funded in whole or in part by the Austrian Science Fund (FWF) \href{https://www.fwf.ac.at/en/research-radar/10.55776/ESP1829225}{10.55776/ESP1829225}. For open access purposes, the authors have applied a CC BY public copyright license to any author accepted manuscript version arising from this submission.
We thank F.~Wehrung for the bibliographic help regarding the concept of breadth for upper semi-lattices, and the anonymous Referee for valuable suggestions.}

\subjclass[2020]{Primary 03E15, Secondary 03E45}
\keywords{Sierpi\'{n}ski's coverings, breadth of lattices, constructibility degrees}

\begin{abstract}
Generalizing a result of Törnquist and Weiss, we study the connection between the existence of \( \varSigma_2^1 \) Sierpi\'{n}ski's coverings of \( \R^n \), and a cardinal invariant of the upper semi-lattice of constructibility degrees known as breadth. 
\end{abstract}
\maketitle
\section{Introduction}

In 1919 Sierpi\'{n}ski proved that \( \CH \), the Continuum Hypothesis, is equivalent to the existence of two sets \( A_0 , A_1 \) covering \( \R^2 \) and such that every line parallel to the \( x \)-axis intersects \( A_0 \) in a countable set, and every line parallel to the \( y \)-axis intersects \( A_1 \) in a countable set.
Three decades later, in 1951, Sierpi\'{n}ski obtained another geometric statement equivalent to \( \CH \): \( \R^3 \) is the union of three sets \( A_0 , A_1 , A_2 \) such that every line parallel to \( \mathbf{e}_i \) has finite intersection with \( A_i \), where \( ( \mathbf{e}_0 , \mathbf{e}_1 , \mathbf{e}_2 ) \) is the canonical basis for \( \R^3 \).
These results were generalized to higher dimensions by Kuratowski and Sierpi\'{n}ski (see~\cite{Simms:1991jk} for a detailed account of the history of these results): for all \( n , k \in \omega \)
\begin{multline}\label{eq:Sierpinskigeneral}
2^{\aleph_0} \leq \aleph_{n + k} \IFF \exists A_0 , \dots , A_{ n + 1 } \bigl ( \R^{ n + 2 } = \bigcup_{ i \leq n + 1 } A_i 
\\
 \text{ and } \FORALL{ i \leq n + 1 } \FORALL{ \ell \in \mathcal{L}_i ( \R^{n + 2} ) } ( \card{ A_i \cap \ell } < \aleph_k ) \bigr ) .
\end{multline}
Here and below \( \mathcal{L}_i ( \R^{n} ) \) is the subset of all lines in \( \R^{n} \) that are parallel to the vector \( \mathbf{e}_i \), the \( i \)-th vector of the canonical basis for \( \R^{n} \).
Setting \( n = 0 \) and \( k = 1 \), or \( n = 1 \) and \( k = 0 \) in~\eqref{eq:Sierpinskigeneral} the two theorems of Sierpi\'{n}ski from 1919 and 1951 are obtained.

Several other geometric statements yield a bound on the size of the continuum.
For example \( 2^{\aleph_0} \leq \aleph_{n + k} \) is equivalent to: ``the number of \( \aleph_k \)-fogs it takes the cover the plane is at most \( n + 2 \)”, and similarly for \( \aleph_k \)-clouds and \( \aleph_k \)-sprays.
(See Section~\ref{subsec:coveringtheplanewithsprays} for the definitions of fog, cloud, and spray.)

The \( A_i \)s in~\eqref{eq:Sierpinskigeneral} are constructed using a transfinite induction of length \( \card{ \R } \), so their descriptive complexity is bounded by the complexity of the well-ordering of \( \R \). 
On the other hand, some of the \( A_i \)s are neither measurable nor have the property of Baire (Lemma~\ref{lem:noBaire}), and therefore they can't be \( \boldsymbol{\Delta}^{1}_{2} \) (provably in \( \ZFC \)). 
In Gödel's constructible universe \( \Ll \), there is a good \( \varSigma^1_2 \) well-ordering of \( \R \), and Törnquist and Weiss in~\cite{Tornquist:2015ys} proved that \( \R \subseteq \Ll \) is equivalent to either one of the following statements:
\begin{gather*}
 \exists A_0 , A_1 \in \varSigma^1_2 \bigl ( A_0 \cup A_1 = \R^2 \text{ and } \forall i < 2 \forall \ell \in \mathcal{L}_i ( \R^2 ) ( \card{ \ell \cap A_i } < \aleph_1 ) \bigr ) 
\\
 \exists A_0 , A_1 , A_2 \in \varSigma^1_2 \bigl ( A_0 \cup A_1 \cup A_2 = \R^3 \text{ and } \forall i < 3 \forall \ell \in \mathcal{L}_i ( \R^3 ) ( \card{ \ell \cap A_i } < \aleph_0 ) \bigr ) .
\end{gather*}

In other words, Törnquist and Weiss showed that by asking the \( A_i \)s to be definable in the best possible way (i.e. \( \varSigma_2^1 \)) in both Sierpi\'{n}ski's 1919 and 1951 results, we get the ``strongest" version of \( \CH \) (i.e. \( \R \subseteq \Ll \)).

In this paper, we generalize their result. 
In Section~\ref{sec:preliminaries}, we introduce all notions involved, i.e.~constructibility degrees, the breadth of an upper semi-lattice and Sierpi\'{n}ski's coverings. 
Then, in Section~\ref{sec:realbreadth}, we study the relationship between the breadth of the upper semi-lattice of constructibility degrees and the size of the continuum (Theorem~\ref{th:realbreadth}). 
In Section~\ref{sec:Definablecover} we generalize Törnquist and Weiss' result by showing that the existence of a \( \varSigma_2^1 \) covering as in~\eqref{eq:Sierpinskigeneral}, with \( n + k > 1 \) and \( k \le 1 \), is equivalent to the breadth of the upper semi-lattice of constructibility degrees having a certain finite upper bound. 
Finally, in Section~\ref{sec:conclusion}, we conclude with some open questions.

\subsection*{Notation} Our notation is standard---see e.g.~\cite{Jech:2003pd}. 
When we treat a transitive set \( M \) as a model-theoretic structure, we use the language of set theory, identifying \( M \) with the structure \( \langle M, \in \rangle \).
We write \( M \prec_1 N \) to mean that \( \langle M, \in \rangle \) is a \( \Sigma_1 \)-elementary substructure of \( \langle N, \in \rangle \).

\section{Preliminaries}\label{sec:preliminaries}

\subsection{Sierpi\'{n}ski coverings}
The theorem of Sierpi\'{n}ski-Kuratowski~\eqref{eq:Sierpinskigeneral} deals with sets \( A_0 , \dots , A_{n-1} \) such that \( \bigcup_{i < n} A_i = \R^n \) and such that \( \ell \cap A_i \) is small, for any line \( \ell \) with direction \( \mathbf{e}_i \).
In the applications below, small means being at most countable, or thin (Lemma~\ref{lem:noBaire}).
Recall that a set of reals is \markdef{thin} if it does not contain a non-empty perfect set.

Sets \( A_i \) as in~\eqref{eq:Sierpinskigeneral} form a \markdef{Sierpi\'{n}ski covering} of \( \R^n \).
As mentioned in the introduction, we let
\[
 \mathcal{L}_i ( \R^n ) \coloneqq \setof{ \ell \subseteq \R^n }{ \ell \text{ is a line parallel to } \mathbf{e}_i } .
\]
A line \( \ell \in \mathcal{L}_i ( \R^n ) \) is uniquely determined by a point \( \mathbf{p} \) in 
\[
 H_i = H_i^n \coloneqq \setof{ ( x_0 , \dots , x_{ n - 1 } ) \in \R^n }{ x_i = 0 } 
\]
the coordinate hyperplane orthogonal to \( \mathbf{e}_i \). 
We denote by \( \ell_\mathbf{p} \) the line determined by \( \mathbf{p} \).

\begin{lemma}\label{lem:noBaire}
Let \( n \in \omega \) and suppose that \( \R^n = \bigcup_{i < n } A_i \) and 
\[
\forall i < n \ \bigl ( \setof{ \mathbf{p} \in H_i }{ \ell_\mathbf{p} \cap A_i \text{ is thin} } \text{ is comeager in } H_i \bigr ) .
\]
Then, the \( A_i \)s cannot all have the property of Baire.
\end{lemma}

\begin{proof}
Suppose otherwise, and fix \( i < n \).
By Kuratowski-Ulam theorem~\cite[Theorem 8.41]{Kechris:1995kc} the set of all the \( \mathbf{p} \in H_i \) such that \( \ell_\mathbf{p} \cap A_i \) has the property of Baire in \( H_i \), is comeager in \( H_i \). 
Therefore, there is a comeager set (in \( H_i \)) of \( \mathbf{p} \)'s such that \( \ell_\mathbf{P} \cap A_i \) is thin and has the property of Baire (in \( H_i \)). 
However, a set that is thin and has the property of Baire is necessarily meager. 
Therefore, for comeager-many \( \mathbf{p} \)'s, the set \( \ell_\mathbf{p} \cap A_i \) is meager. 
It follows again from Kuratowski-Ulam theorem that the set \( A_i \) is meager in \( \R^n \). 
However, this leads to a contradiction with Baire’s theorem~\cite[Theorem 8.4]{Kechris:1995kc}, as it implies that \( \R^{n} \), being the union of finitely many meager sets, is meager.
\end{proof}

A similar result holds with Lebesgue-measurability in place of the property of Baire.

\subsection{Constructibility degrees}\label{sec:constructibility}
The \markdef{constructibility preorder} \( \leqc \) on \( \Pow ( \omega ) \) is defined by 
\[
 a \leqc b \text{ iff } a \in \Ll [ b ] \text{, which is equivalent to } \Ll [ a ] \subseteq \Ll [ b ] .
\]
Let \( \eqc \) be the induced equivalence relation, that is 
\[ 
 a \eqc b \text{ iff } a \leqc b \leqc a \text{, which is equivalent to } \Ll [ a ] = \Ll [ b ] .
\]
The quotient \( \Pow ( \omega ) / {\eqc} \) with the induced order is the set of \markdef{constructibility degrees}, and it is denoted by \( \constructibledegrees \). 
If \( \mathbf{a} \in \constructibledegrees \), then set \( \Ll [ \mathbf{a} ] \) to be \( \Ll [ a ] \) for some/any \( a \in \mathbf{a} \).
Note that \( \constructibledegrees \) is a partial order with minimum \( \mathbf{0} = {\Pow ( \omega  ) } \cap {\Ll} \), every degree \( \mathbf{a} \) has size \( ( \aleph_1 )^{ \Ll [ \mathbf{a} ] } \le \aleph_1 \), and \( \setof{ \mathbf{b} }{ \mathbf{b} \leqc \mathbf{a} } \), the set of all predecessors of \( \mathbf{a} \in \constructibledegrees \) has size \( \leq ( \aleph_1 )^{ \Ll [ \mathbf{a} ] } \leq \aleph _1 \). 
Our official definition of \( \leqc \) takes place on \( \Pow ( \omega ) \), but since the Cantor space and the Euclidean line \( \R \) are \( \varDelta^1_1 \) isomorphic~\cite[Theorem 3E.7]{MR2526093}, we could have used the true reals without any problem.
The order \( \constructibledegrees \) is an upper semi-lattice as \( \mathbf{a} \vee \mathbf{b} = \eq{a \oplus b}_{\mathrm{c}} \) for some/any \( a \in \mathbf{a} \) and \( b \in \mathbf{b} \), where 
\[ 
a \oplus b = \setof{ 2 n }{ n \in a } \cup \setof{ 2 n + 1 }{ n \in b } . 
\]
More generally, if \( \langle \cdot , \cdot \rangle \colon \omega \times \omega \to \omega \) is a recursive bijection, and \( a_n \subseteq \omega \), then letting
\[
\textstyle\bigoplus_{ n \in \omega } a_n \coloneqq \setof{ \langle n , k \rangle }{ n \in \omega \text{ and } k \in a_n }
\]
we have that \( a_i \leqc \bigoplus_{ n \in \omega } a_n \) for all \( i < \omega \).
Thus the order \( \constructibledegrees \) is \( \sigma \)-bounded---any countable set of elements \( \setof{ \eq{a_n}_{\mathrm{c}} }{ n \in \omega } \) has an upper bound \( \eq{ \bigoplus_{ n \in \omega} a_n }_{\mathrm{c}} \).
On the other hand \( \setof{ \eq{ a_n }_{\mathrm{c}} }{ n \in \omega } \) need not have a least upper bound, i.e. \( \constructibledegrees \) need not be a \( \sigma \)-complete upper semi-lattice~\cite{Truss:1978aa}. 
Again in~\cite{Truss:1978aa}, it is also shown that \( \constructibledegrees \) needs not be a lower semi-lattice.

The structure of \( \constructibledegrees \) is highly non-absolute, being very sensitive to the ambient model. 
If \( \Vv = \Ll \), then \( \constructibledegrees \) is the singleton \( \set{ \mathbf{0} } \), while if \( \Vv = \Ll [ r ] \) where \( r \) is a Cohen real, then \( \constructibledegrees \) has the size of the continuum and a rich structure~\cite{Abraham:1986aa}.

Adamowicz~\cite{Smolska-Adamowicz:1977aa} has shown that for every constructible, constructibly countable and well-founded upper semi-lattice with a lowest element, there is a generic extension of the constructible universe in which \( \constructibledegrees \) is isomorphic to the given upper semi-lattice---see~\cite{Groszek:1988aa,Shore:1992wm} for stronger results and more discussion on this. 

Even though there is not yet a general result \emph{à la} Adamowicz for uncountable upper semi-lattices, there are some scattered, yet interesting, results in this regard: in the iterated perfect set model, also known as iterated Sacks model, the continuum is \( \aleph_2 \) and \( \constructibledegrees \) is well-ordered of order-type \( \omega _2 \)~\cite{Baumgartner:1979vo,Groszek:1981aa}; Groszek~\cite{Groszek:1994aa} has shown that \( \constructibledegrees \) can, consistently with \( \ZFC \), be isomorphic to the reverse copy of \( \omega_1 + 1 \); see also~\cite{Kanovei:1998aa,Kanovei:1999aa} for similar results. 

\subsection{Breadth of an upper semi-lattice}\label{sec:breadth}

Let \( ( P, \preceq ) \) be a preorder and \( F\subseteq P \), we denote by \( {\downarrow} F \) the set of the predecessors of elements of \( F \), i.e. the set \( \setof{ q \in P }{ q \preceq p \text{ for some } p \in F } \). 
Given a \( p \in P \) we write \( {\downarrow} p \) instead of \( {\downarrow} \set{ p } \). 
Two distinct elements \( p , q \in P \) are incomparable if neither \( x \preceq y \) nor \( y \preceq x \); an antichain of \( P \) is a subset of pairwise incomparable elements.
The next result is folklore.

\begin{proposition}\label{prop:boundonantichains}
Suppose \( ( P , \preceq ) \) is a preorder such that \( \card{ { \downarrow } p } < \aleph_ \alpha \) for all \( p \in P \).
If every antichain has size \( \leq \aleph_ \beta \), then \( \card{P} \leq \max (  \aleph_ \alpha ,  \aleph_ \beta ) \). 
\end{proposition}

\begin{proof}
Construct by induction \( I_ \nu \subseteq P \) for \( \nu < \aleph_ \alpha \) so that \( I_ \nu \) is a maximal antichain of \( P \setminus \bigcup_{ \xi  < \nu } {\downarrow} I_ \xi \).
(Some of the \( I_ \nu \)s maybe empty, and if \( I_\nu = \emptyset \) then \( I_ \xi = \emptyset \) for all \( \nu < \xi < \aleph_ \alpha  \).)
We argue that \( P = \bigcup_{ \nu < \aleph_ \alpha } {\downarrow} I_ \nu \), so that \( \card{P} \leq \sum_{ \nu < \aleph_ \alpha } \card{ {\downarrow} I_ \nu } \leq \aleph_ \beta  \cdot \aleph_ \alpha \).
Towards a contradiction, suppose there is \( \bar{p} \in P \setminus \bigcup_{ \nu < \aleph_ \alpha } { \downarrow } I_ \nu \).
Then \( I_ \nu \neq \emptyset \) for each \( \nu < \aleph_ \alpha \), so by maximality there is \( q_ \nu \in I_ \nu \) comparable with \( \bar{p} \); as \( \bar{p} \notin { \downarrow } I_ \nu \), then \( q_ \nu \preceq \bar{ p } \).
As the \( I_ \nu \)s are disjoint, then the \( q_ \nu \)s are distinct, and they all belong to \( { \downarrow } p \), against our assumption on \( P \).
\end{proof}

If \( ( P , \preceq ) \) is an upper semi-lattice and \( F \) is a finite subset of \( P \), we denote by \( \bigvee F \) the least upper bound of \( F \). 
All upper semi-lattices considered here have a least element \( \mathbf{0} \), so \( \bigvee \emptyset = \mathbf{0} \).
A subset \( I \subseteq P \) is called an \markdef{ideal} if it is a downward-closed (i.e. \( {\downarrow} I = I \) ) sub-semi-lattice of \( P \); it is proper if \( I \neq P \).

\begin{definition}[Ditor, {\cite[\S 4]{Ditor:1984aa}}]
Given an upper semi-lattice \( ( P , \preceq ) \), a finite subset \( F \subseteq P \) is \markdef{redundant} if there is a \( p \in F \) such that \( p \preceq \bigvee F \setminus \set{ p } \), otherwise we call it \markdef{irredundant}.

The \markdef{breadth} of \( P \) is the supremum of the cardinalities of irredundant subsets of \( P \).
If \( p \in P \), we say that \( P \) \markdef{has breadth at most \( \kappa \) above \( p \)} if the sub-semi-lattice \( \setof{ q \in P }{ p \preceq q } \) has breadth at most \( \kappa \).
\end{definition}

Observe that if \( P \) has breadth \( n \), then \( n = 0  \) if and only if  \( P  \) is a singleton \( \{ \mathbf{0} \} \), and \( n = 1 \) if and only if  \( P \) is a linearly ordered set with at least two elements, and hence \( n \geq 2 \) if and only if \( P \) has antichains of size at least \( 2 \).

The notion of breadth yields much sharper bounds than the size of the antichains.

\begin{theorem}[{\cite{Ditor:1984aa}}]\label{th:Ditor}
If \( P \) is an upper semi-lattice of breadth \( n > 0 \) and such that \( \card{ {\downarrow} p } < \aleph_{\alpha} \) for all \( p \in P \), then 
\begin{enumerate}[label={\upshape (\alph*)}]
\item\label{th:Ditor-1}
\( \card{ P } \le \aleph_{ \alpha + n - 1 } \), and
\item \label{th:Ditor-2}
\( \card{ I } < \aleph_{ \alpha + n - 1 } \) for every proper ideal \( I \) of \( P \).
\end{enumerate}
\end{theorem}

Ditor also showed that this cardinal bound is sharp when \( n = 2 \) and \( \aleph_\alpha \) is regular.

\begin{theorem}[{\cite{Ditor:1984aa}}]
For every regular \( \aleph_{ \alpha } \), there exists an upper semi-lattice \( P \) of breadth \( 2 \) such that \( \card{ P } = \aleph_{ \alpha + 1 } \) and \( \card{ {\downarrow} p } < \aleph_{ \alpha } \) for every \( p \in P \).
\end{theorem}

Then, Wehrung improved the latter result by showing that the cardinal bound of Theorem~\ref{th:Ditor} is sharp also for every \( n > 2 \) and every uncountable regular \( \aleph_{ \alpha } \).

\begin{theorem}[{\cite{Wehrung:2010aa}}]
For every uncountable regular \( \aleph_{ \alpha } \), for every \( n > 0 \), there exists an upper semi-lattice \( P \) of breadth \( n \) such that \( \card{ P } = \aleph_{ \alpha + n - 1 } \) and \( \card{ { \downarrow } p } < \aleph_{ \alpha } \) for every \( p \in P \).
\end{theorem}

See  Section~\ref{sec:largesacks} for more discussion on these resuls and related open questions.

\section{The breadth of constructibility degrees}\label{sec:realbreadth}

In this section, we explore the relationship between the breadth of the upper semi-lattice of constructibility degrees and the size of the continuum. 
Consider the following statement:
\begin{equation*}\label{eq:mincont}\tag{\( \star \)}
\forall M \prec_1 \Hh_{\omega_1} \left ( \card{ M } = 2^{ \aleph_0 } \Rightarrow M = \Hh_{\omega_1 } \right ) ,
\end{equation*}
where  \( \Hh_{ \omega_1 } \coloneqq \setof{ x }{ \card{ \TC ( x ) } < \omega _1 } \) is  the  set of all hereditarily countable sets.

The property \eqref{eq:mincont} may be interpreted as a kind of minimality principle: as soon as a \( \Sigma_1 \)-elementary substructure of \( \langle \Hh_{ \omega_1}, \in \rangle \) has size of the continuum, it already coincides with \( \Hh_{ \omega_1 }\). 
This interpretation is supported by the following result.

\begin{proposition}\label{prop:nocohen}
Assume \eqref{eq:mincont}. 
Then, there are no \( \Ll \)-generic Cohen reals.
\end{proposition}

We need the following theorem, known as \emph{Lévy-Shoenfield Absoluteness Theorem}.

\begin{theorem}[Lévy-Shoenfield {\cite[Theorem 36]{Jech:1978aa}}]
For every  \( a \in \R \), if \( \theta = \omega_1^{ \Ll [ a ] } \), then 
\[
\Ll_{ \theta } [ a ] \prec_1 \Hh_{\omega_1}.
\]
\end{theorem}

\begin{proof}[Proof of Proposition~\ref{prop:nocohen}]
Suppose otherwise towards a contradiction. 
Then, it is well-known that there exists a perfect closed set \( \mathcal{C}\subseteq \mathbb{R} \) such that  any finite \( F \subseteq \mathcal{C} \) is a set of mutually Cohen generic reals over \( \Ll \). 
It follows that \( \omega_1^{ \Ll [ F ] } = \omega_1^{\Ll} \), for any such \( F \). 
Now fix an \( x \in \mathcal{C} \) and consider the transitive set
\[
M = \bigcup_{F \in [ \mathcal{C}\setminus \set{x} ]^{<\omega}} \Ll_{\omega_1^{\Ll}}[F].
\]
By Lévy-Shoenfield absoluteness theorem, \( \Ll_{\omega_1^{\Ll}}[F] = \Ll_{\omega_1^{\Ll [ F ] }} [ F ] \prec_1 \Hh_{\omega_1} \) for every \( F \in [ \mathcal{C} ]^{ < \omega } \). 
Therefore, \( M \prec_1 \Hh_{ \omega_1} \). 
Clearly \( \card{ M } = \card{ \mathcal{C} } = 2^{\aleph_0} \), but \( x \notin M \), which contradicts~\eqref{eq:mincont}.
\end{proof}

We next explore the relationship between~\eqref{eq:mincont} and the breadth of the constructibility degrees. 
We need the following well-known fact.

\begin{lemma}\label{lem:cardelem} 
For every \( M \prec_1 \Hh_{\omega_1} \) the following hold:
\begin{enumerate}[label={\upshape (\alph*)}]
\item\label{lem:cardelem-1}
\( M \) is transitive.
\item \label{lem:cardelem-2}
\( \card{ M } = \card{ M \cap \R } \).
\item \label{lem:cardelem-3}
If \( \R \subseteq M \), then \( M = \Hh_{\omega_1} \).
\end{enumerate}
\end{lemma}

\begin{proof}
\ref{lem:cardelem-1}.
Since \( M\prec_1 \Hh_{\omega_1} \), we must have \( \omega \in M \). 
Moreover, since \( \Hh_{ \omega_1 } \Models \)``Every set is countable'', the same sentence holds in \( M \). 
Fix a nonempty \( x \in M \), then there is a function \( f \in M \) such that 
\[
M \Models f \colon \omega \to x \text{ is a surjection}.
\] 
Hence the function \( f \) is a surjection of \( \omega \) onto \( x \) also according to \( \Hh_{\omega_1} \) (and \( \Vv \)). 
Therefore \( x\subseteq M \). 
Thus, \( M \) is transitive.

\smallskip

\ref{lem:cardelem-2}, \ref{lem:cardelem-3}.
It is well known (e.g. see~\cite[\S VII.3]{Simpson:1999aa}, or~\cite[Lemma 25.25]{Jech:2003pd}, or~\cite[Proposition 13.8]{Kanamori:2003zk}) that the map \( G \) which, given any real that recursively encodes an extensional and well-founded relation on \( \omega \), returns the Mostowski collapse of such relation, is \( \Delta_1 \)-definable over \( \Hh_{\omega_1} \). 
Moreover, 
\[
\Hh_{\omega_1} \Models \FORALL{x} \EXISTS{y} ( y\subseteq \omega \text{ and } G ( y ) = x ).
\]
Since \( M \prec_1 \Hh_{ \omega_1 } \), the same sentence is satisfied by \( M \), thus \ref{lem:cardelem-2} and \ref{lem:cardelem-3} follow.
\end{proof}

\begin{proposition}\label{prop:RL}
The following are equivalent:
\begin{enumerate}[label={\upshape (\alph*)}]
\item \label{prop:RL-1}
\( \R \subseteq \Ll \).
\item \label{prop:RL-2}
\( \CH + \eqref{eq:mincont} \).
\end{enumerate}
\end{proposition}

\begin{proof}
\ref{prop:RL-1}\( \Rightarrow \)\ref{prop:RL-2}. 
Since \( \R\subseteq \Ll \), then \( \Hh_{\omega_1} = \Ll_{\omega_1} \) and \( \CH \) holds. 
Fix an \( M \prec_1 \Ll_{ \omega_1 } \). 
By a direct corollary of Gödel's Condensation Lemma~\cite[Lemma 5.10]{Devlin:1984aa}, \( M = \Ll_{\alpha} \) for some \( \alpha \le \omega_1 \). 
When \( \alpha < \omega_1, \ M \) is countable, otherwise it coincides with \( \Ll_{\omega_1} = \Hh_{\omega_1} \). 
Therefore~\eqref{eq:mincont} holds.

\smallskip

\ref{prop:RL-2}\( \Rightarrow \)\ref{prop:RL-1}. 
Note that we must have \( \omega_1^{\Ll} = \omega_1 \), as otherwise there would exists an \(\Ll\)-generic Cohen real, against Proposition~\ref{prop:nocohen}. 
By Lévy-Shoenfield, \( \Ll_{\omega_1} \prec_1 \Hh_{\omega_1} \). 
By \( \CH \), \( \Ll_{\omega_1 } \) has size the continuum, and therefore, by~\eqref{eq:mincont}, \( \Ll_{\omega_1} = \Hh_{\omega_1} \). 
Hence \( \R\subseteq \Ll \).
\end{proof}

Note that the statement \(\R \subseteq \Ll\) is equivalent to saying that \( \constructibledegrees \) has breadth \( 0 \). 
A direct consequence of Theorem~\ref{th:Ditor}\ref{th:Ditor-1} applied to the upper semi-lattice of constructibility degrees is the following:

\begin{proposition}\label{prop:realbreadth}
Suppose that \( \constructibledegrees \) has breath at most \( n \) for some \( n \in \omega \). 
Then, \( 2^{\aleph_0} \le \aleph_{ n + 1 } \).
\end{proposition}

However, using the definability of the constructibility preorder, we can say something more (cf. Proposition~\ref{prop:RL}).

\begin{theorem}\label{th:realbreadth}
Suppose that \( \constructibledegrees \) has breath at most \( n \) for some \( n \in \omega \). 
Then, either \( 2^{\aleph_0} \le \aleph_n \) or~\eqref{eq:mincont}.
\end{theorem}

\begin{proof}
We already proved the case \( n = 0 \) in Proposition~\ref{prop:RL}. 
So fix an \( n > 0 \) and suppose that \( 2^{ \aleph_0 } = \aleph_{ n + 1 } \) and that \( \constructibledegrees \) has breadth \( n \), towards proving~\eqref{eq:mincont}. 

Fix some \( M \prec_1 \Hh_{\omega_1} \). 
By part~\ref{lem:cardelem-1} of Lemma~\ref{lem:cardelem}, \( \omega_1 \cap M  \) is an ordinal \( \alpha \leq \omega _1 \).
Let \( ( \constructibledegrees ^M , \leqc^M ) \) be the upper semi-lattice of constructibility degrees relativized to \( M \).
For every \( x , y \in \R \cap M \), we have
\begin{equation}\label{eq:minimalbreadth1}
 x \leqc^M y \IFF x \in \Ll_{\alpha} [ y ] .
\end{equation}
There are two cases:
\begin{description}
\item[Case 1]
\( \alpha < \omega_1 \).
As the sentence ``\( \constructibledegrees \) has breadth at most \( n \)" is \( \Pi_2 \) over \( \Hh_{\omega_1} \) and \( M \prec_1 \Hh_{\omega_1} \), we have that \( \constructibledegrees ^M \) has breadth at most \( n \).
By~\eqref{eq:minimalbreadth1} and case assumption,  any element of \(\constructibledegrees ^M\) has at most countably many \( \leqc^M \)-predecessors. 
By part~\ref{th:Ditor-1} of  Theorem~\ref{th:Ditor}, we have \( \card{ \constructibledegrees ^M } \le \aleph_n \) and therefore \( \card{ \R \cap M } \le \card{ \constructibledegrees ^M }\cdot \aleph_0 \le \aleph_n \). 
By part~\ref{lem:cardelem-2} of Lemma~\ref{lem:cardelem}, we conclude that \( \card{ M } \le \aleph_n < 2^{ \aleph_0 } \).
\item[Case 2]
\( \alpha = \omega_1 \).
If \( x , y \in \R \) and \( y \in M \), and \( x \leqc y \), then \( x \in \Ll_{ \omega _1 } [ y ] \) so \( x \leqc^M y \) by~\eqref{eq:minimalbreadth1}. 
In particular, \( \constructibledegrees ^M \) is an ideal of \( \constructibledegrees \). 
By part~\ref{th:Ditor-2} of Theorem~\ref{th:Ditor}, if \( \constructibledegrees ^M \) is a proper ideal, then \( \card{ \constructibledegrees ^M } \le \aleph_n \). 
In this case \( \card{ \R \cap M } \le \card{ \constructibledegrees ^M } \cdot \aleph_1 \le \aleph_n \) and therefore \( \card{ M } \le \aleph_n < 2^{ \aleph_0 } \).
Otherwise, we would have \( \R \subseteq M \) and then \( M = \Hh_{\omega_1} \).
\end{description}
Either way~\eqref{eq:mincont} holds. 
\end{proof}

\begin{remark}\label{rmk:breadth}
Theorem~\ref{th:realbreadth} is a meaningful extension of Proposition~\ref{prop:realbreadth} only if it is consistent that \( 2^{ \aleph_0 } = \aleph_{ n + 1 } \) and that \( \constructibledegrees \) has breadth \( n \).
The natural question is whether this assumption is consistent for every \( n > 0 \). 
We know that for \( n = 1 \) this is the case, with the iterated perfect set model witnessing the consistency (see Section~\ref{sec:constructibility}). 
It is open whether it is the case also for \( n > 1 \) (see Section~\ref{sec:largesacks}).
\end{remark}

\begin{remark}
In the light of Proposition~\ref{prop:RL}, one may be tempted to think that \( 2^{\aleph_0}= \aleph_{ n + 1 } + \eqref{eq:mincont} \) may imply that \( \constructibledegrees \) has breadth at most \( n \). 
This is certainly true when \( n = 0 \) (which is the content of Proposition~\ref{prop:RL}), but it fails badly already for \( n=1 \). 
Indeed, let \( \forcing{Q} \) be the countable-support iteration \( \seqof{ \forcing{Q}_\alpha }{ \alpha < \omega_2 } \) such that \( \forcing{Q}_\alpha \) is forced to be \( \forcing{S}\times \forcing{S} \), the product of two Sacks forcing---see~\cite[Proposition 2.4]{Groszek:2019aa} for some general properties of this kind of forcing. 
If we let \( G \) be a \( \forcing{Q} \)-generic filter over \( \Ll \), then it can be shown that, in \( \Ll [ G ] \), the continuum is \( \aleph_2 \) and~\eqref{eq:mincont} holds, but the breadth of \( \constructibledegrees \) in \( \Ll [ G ] \) is \( 2 \), cofinally---i.e. \( \constructibledegrees \) has breadth \( 2 \) above \( \mathbf{a} \) for every \( \mathbf{a} \in \constructibledegrees\).
\end{remark}

\section{Definable Sierpi\'{n}ski coverings}\label{sec:Definablecover}

The next theorem shows the connection between the existence of \( \varSigma_2^1 \) Sierpi\'{n}ski coverings and the breadth of the upper semi-lattice of the constructibility degrees. 
The case \( n = 0 \) has already been shown by Törnquist and Weiss in~\cite{Tornquist:2015ys}---see the introduction.

\begin{theorem}\label{th:main}
For every \( n \in \omega \), the following are equivalent:
\begin{enumerate}[label={\upshape (\alph*)}]
\item\label{th:main-1} 
\( \constructibledegrees \) has breadth at most \( n \).
\item\label{th:main-2}
There are \( \varSigma_2^1 \) sets \( A_0 , \dots , A_{ n + 1 } \subseteq \R ^{ n + 2 } \) such that \( \R ^{ n + 2 } = \bigcup_{ i \le n } A_i \) and for all \( i \le n + 1 \), for all \( \ell \in \mathcal{L}_i ( \R ^{ n + 2 } ) \), \( \ell\cap A_i \) is countable.
\item\label{th:main-3}
There are \( \varSigma_2^1 \) sets \( A_0, \dots, A_{ n + 2 } \subseteq \R ^{n + 3} \) such that \( \R ^{ n + 3 } = \bigcup_{ i \le n + 2 } A_i \) and for all \( i \le n + 2 \), for all \( \ell \in \mathcal{L}_i ( \R ^{ n + 3} ) \), \( \ell \cap A_i \) is finite.
\end{enumerate}
\end{theorem}

Before delving into its proof, we need the following lemma.

\begin{lemma}\label{lemma:satisf}
Given a \( \Sigma_1 \) formula \( \upvarphi ( x , y ) \) in the language of set theory, the set 
\[
 \setofLR{ ( x , y ) \in \R ^2 }{ x \leqc y \text{ and } \Ll [ y ] \Models \upvarphi ( x , y ) }
\]
is \( \varSigma_2^1 \).
\end{lemma}

\begin{proof}
It suffices to prove that our set is \( \Sigma_1 \) over \( \Hh_{\omega_1} \)~\cite[Lemma 25.25]{Jech:2003pd}. 
In other words, we need to show that there is a \( \Sigma_1 \) formula \( \uppsi ( x , y ) \) in the language of set theory such that 
\[
\bigl ( x \leqc y \text{ and } \Ll [ y ] \Models \upvarphi ( x , y ) \bigr ) \IFF \Hh_{ \omega_1 } \Models \uppsi ( x , y ) .
\]
By Gödel's Condensation Lemma, for every \( x \in \R \cap \Ll [ y ] \), 
\[
\Ll [ y ] \Models \upvarphi ( x , y ) \IFF \Ll_\delta [ y ] \Models \upvarphi ( x , y ) \text{, for some countable ordinal } \delta.
\]
Then, for any \( x , y \in \R \),
\begin{equation}\label{eq:2}
\left ( x \leqc y \text{ and } \Ll [ y ] \Models \upvarphi ( x , y ) \right ) \IFF \Hh_{\omega_1} \Models \exists \delta \bigl ( x \in \Ll_\delta [ y ] \text{ and } \Ll_\delta [ y ] \Models \upvarphi ( x , y ) \bigr ) .
\end{equation}
The sentence ``\( x \in \Ll_\delta [ y ] \)'' is \( \Delta_1 ( x , y, \delta ) \) over \( \Hh_{ \omega_1 } \). 
Regarding the sentence ``\( \Ll_\delta [ y ] \Models \upvarphi ( x , y ) \)", it does not matter if we interpret it as a genuine satisfaction relation~\cite[Ch.~1, \S 9]{Devlin:1984aa} or as a relativization~\cite[Definition 12.6]{Jech:2003pd}, because in both cases, the complexity of the sentence is at most \( \Delta_1 ( x , y , \delta ) \) over \( \Hh_{ \omega_1 } \). 
Hence our set is \( \Sigma_1 \) over \( \Hh_{\omega_1} \).
\end{proof}

We also need the following theorem.

\begin{theorem}[Mansfield-Solovay, {\cite[Corollary 14.9]{Kanamori:2003zk}}]\label{th:Mansfield}
If \( X \subseteq \R \) is \( \varSigma^1_2 ( c ) \) and \( X \nsubseteq \Ll [ c ] \), then \( X \) contains a nonempty perfect set.
\end{theorem}

Now, we are ready to prove our main theorem. 

\begin{proof}[Proof of Theorem \ref{th:main}]
\ref{th:main-1}\( \Rightarrow \)\ref{th:main-2}. 
For each \( i \le n + 1 \), let \( A_i \) be the set 
\[
\setof{ ( x_0, \dots, x_{ n + 1 } ) \in \R ^{ n + 2 } }{ \textstyle\EXISTS{ j \neq i} ( x_i, x_j \leqc \bigoplus_{ k \neq i , j } x_k \text{ and } \Ll \bigl [ \bigoplus_{ k \neq i , j } x_k \bigr ] \Models x_i \unlhd x_j ) }
\]
where \( \unlhd \) is the canonical \( \Sigma_1 \) well-ordering of \( \Ll [ \bigoplus_{ k \neq i , j } x_k ] \). 
This is a \( \varSigma_2^1 \) definition by Lemma~\ref{lemma:satisf}. 

Next, we show that the \( A_i \)'s cover \( \R ^{ n + 2 } \) and that for each \( i \le n + 1 \) and for each line \( \ell \in \mathcal{L}_i ( \R ^{ n + 2 } ) \), \( \ell \cap A_i \) is countable. 
Pick any \( ( x_0 , \dots , x_{ n + 1 } ) \in \R ^{ n + 2 } \). 
As, by hypothesis, \( \constructibledegrees \) has breadth at most \( n \), it follows that there are distinct \( i , j \le n + 1 \) such that \( x_i, x_j \leqc \bigoplus_{ k \neq i , j } x_k \). 
Now, either \( \Ll [ \bigoplus_{ k \neq i , j } x_k ] \Models x_i \unlhd x_j \), and then, by definition, \( ( x_0 , \dots, x_{ n + 1 } ) \in A_i \), or \( \Ll [ \bigoplus_{ k \neq i , j } x_k ] \Models x_i \rhd x_j \), and then \( ( x_0, \dots, x_{n + 1} ) \in A_j \). 
Thus \( \R ^{ n + 2 } = \bigcup_{ i \le n + 1 } A_i \).

Fix an \( i \le n + 1 \) and an \( n + 1 \)-tuple \( ( x_0, \dots, x_{n} ) \in \R ^{ n + 1 } \). 
By definition, for each \( y \in \R \),
\begin{multline*}
( x_0 , \dots, x_{ i - 1 } , y , x_{i} , \dots, x_{n} ) \in A_i \IFF {}
 \\
\textstyle \EXISTS{ j \le n }\Bigl ( y , x_j \leqc \bigoplus_{ k \neq j } x_k \text{ and } \Ll \bigl [ \bigoplus_{ k \neq j } x_k \bigr ] \Models y \unlhd x_j \Bigr ) .
\end{multline*}
Since the choices of the \( j \)s in the formula above are finite, and each initial segment of \( \unlhd \restriction \R \) is countable, it follows that the set 
\[
 \setof{ y \in \R }{ ( x_0 , \dots, x_{ i - 1 } , y , x_{i}, \dots, x_{n} ) \in A_i }
\] 
is countable.

\smallskip

\ref{th:main-1}\( \Rightarrow \)\ref{th:main-3}.
For each \( i \le n + 2 \) let \( A_i \) be the set 
\begin{multline*}
 \bigl \{ ( x_0, \dots, x_{ n + 2 } ) \in \R ^{ n + 3 } \Mid \EXISTS{ j , l \neq i } ( j \neq l \text{ and } x_i , x_j , x_l \leqc \textstyle \bigoplus_{ k \neq i , j , l } x_k \text{ and} 
 \\
\shoveright{ \textstyle\Ll \bigl [ \bigoplus_{ k \neq i , j , l } x_k \bigr ] \Models ``x_i , x_j \unlhd x_l \text{ and } f ( x_i ) \le f ( x_j ) \text{, where } f \text{ is the }} 
\\
 \unlhd \text{-least bijection between the } \unlhd \text{-predecessors of } x_l \text{ and \( \omega \)''} \bigr \}
\end{multline*}
where, as before, \( \unlhd \) is the canonical \( \Sigma_1 \) well-ordering of \( \Ll [ \oplus_{ k \neq i , j , l } x_k ] \). 
These are \( \varSigma_2^1 \) sets by Lemma~\ref{lemma:satisf}. 

Arguing as in case~\ref{th:main-1}\( \Rightarrow \)\ref{th:main-2}, it follows that \( \R ^{ n + 3 } = \bigcup_{ i \le n + 2 } A_i \) and that for any \( i \le n + 2 \), for any \( ( x_0 , \dots, x_{ n + 1 } ) \in \R ^{ n + 2 } \), the set 
\[
\setof{ y \in \R }{ ( x_0 , \dots , x_{ i - 1 } , y , x_{i}, \dots, x_{ n + 1 } ) \in A_i }
\] 
is finite.

\smallskip

\ref{th:main-2}\( \Rightarrow  \)\ref{th:main-1}.
Towards a contradiction, let \( \set{ \eq{b_0}_{\mathrm{c}} , \dots , \eq{b_n}_{\mathrm{c}} } \) be an irredundant set of size \( n + 1 \), where \( b_0 , \dots , b_{n} \in \R \). (If \( n = 0 \) this reads as: let \( b_0 \in \R \) be such that \( b_0 \not \in \Ll \).)  
Fix an \( i \le n \). 
For \( u_{ i + 1 } , \dots , u_{ n + 1 } \in \Ll [ b_{ i + 1}, \dots, b_{ n } ] \), the set 
\[
X_i ( u_{ i + 1} , \dots , u_{n + 1} ) = \setof{ ( b_0 , \dots , b_{ i - 1 } , y , u_{ i + 1} , \dots , u_{n + 1} ) }{ y \in \R } \cap A_i 
\]
is countable by assumption, and \( \varSigma^1_2 \)-definable with parameters in \( \Ll [ \bigoplus_{ k \neq i } b_k] \). 
A straightforward consequence of Theorem~\ref{th:Mansfield} is that \( X ( u_{ i + 1 } , \dots , u_{ n + 1 } ) \subseteq \Ll [ \bigoplus_{ k \neq i } b_{k} ] \).
As \( b_i \notin \Ll [ \bigoplus_{ k \neq i } b_k ] \), it follows that 
\[
( b_0 , \dots , b_i , u_{ i + 1} , \dots , u_{n + 1} ) \notin A_i .
\]
Therefore,
\begin{equation*}
 \FORALL{ i \le n } \FORALL{ u_{ i + 1} , \dots , u_{n + 1} \in \Ll [ b_{ i + 1 } , \dots , b_n] } \bigl ( ( b_0 , \dots , b_i , u_{ i + 1 } , \dots , u_{ n + 1 } ) \notin A_i \bigr ) .
\end{equation*}
As \( \Ll [ b_{ i + 1 } , \dots , b_n] \subseteq \Ll [ b_i , \dots , b_n ] \) for all \( i \le n \), we have that
\[
 \FORALL{ i \le n } \forall { u_{ i + 1 } , \dots , u_{ n + 1 } \in \Ll [ b_{ i + 1 } , \dots , b_n ] } \bigl ( ( b_0 , \dots , b_i , u_{ i + 1 } , \dots , u_{ n + 1 } ) \notin \textstyle \bigcup_{ k \le i } A_k \bigr ).
\]
In particular, when \( i = n \), we have \( \FORALL{ u_{ n + 1 } \in \Ll } \bigl ( ( b_0 , \dots , b_{ n } , u_{n + 1} ) \notin \bigcup_{k \le n} A_k \bigr ) \). 
Since, by hypothesis the \( A_i \)s cover \( \R ^{ n + 2 } \), it follows that
\[
 \FORALL{ u_{ n + 1 } \in \Ll } \bigl ( ( b_0 , \dots , b_{ n } , u_{ n + 1 } ) \in A_{ n + 1 } \bigr ) .
\]
If \( \omega _1^{\Ll } = \omega _1 \), then this would imply that the line determined by \( ( b_0 , \dots , b_{ n } ) \) intersects \( A_{ n + 1 } \) in an uncountable set, against our assumption.
If \( \omega _1^{\Ll} < \omega _1 \), then there is \( r \), a Cohen real  over \( \Ll \).
Note that \( \omega _1^{\Ll} = \omega _1^{\Ll [ r ] } \) and that, in \( \Ll [ r ] \), the breadth of \( \constructibledegrees \) is infinite. 
By Shoenfield, \( \R ^{n + 2} \cap \Ll [ r ] = \bigcup_{i < n + 2}\bar{A}_i \) where \( \bar{A}_i = A_i \cap \Ll [ r ] \), and for \( i \le n \)
\[
\Ll [ r ] \Models \bar{A}_i \in \varSigma^{1}_{2} \text{ and } \forall \ell \in \mathcal{L}_{i} ( \ell \cap \bar{A}_i \text{ is countable} ). 
\]
Replacing \( \Vv \) with \( \Ll [ r ] \) the argument above can be repeated reaching a contradiction.

\smallskip

\ref{th:main-3}\( \Rightarrow \)\ref{th:main-1}: 
Towards a contradiction, let \( \set{ \eq{b_0}_{\mathrm{c}} , \dots , \eq{b_n}_{\mathrm{c}} } \) be an irredundant set of size \( n + 1 \).
Fix an \( i\le n \). 
Arguing as before, for all \( u_{ i + 1} , \dots , u_{ n + 2 } \in \Ll [ b_{ i + 1} , \dots , b_{n} ] \) the set 
\[
X_i ( u_{ i + 1} , \dots , u_{ n + 2} ) = \setof{ ( b_0 , \dots , b_{ i - 1 } , y , u_{ i + 1} , \dots , u_{ n + 2 } ) }{ y \in \R } \cap A_i 
\]
is finite by hypothesis, and \( \varSigma^1_2 \)-definable with parameters in \( \Ll [ \bigoplus_{ k \neq i } b_{ k } ] \). 
Since \( b_i \notin \Ll [ \bigoplus_{ k \neq i }b_{ k } ] \), it follows that 
\[
 \FORALL{ i \le n } \FORALL{ u_{ i + 1} , \dots , u_{ n + 2 } \in \Ll [ b_{ i + 1 } , \dots , b_n ] } \bigl ( ( b_0 , \dots , b_i , u_{ i + 1 } , \dots , u_{ n + 2 } ) \notin \textstyle \bigcup_{ k \le i } A_k \bigr ) .
\]
In particular, \( \FORALL{ u_{n + 1} , u_{n+2} \in \Ll } \bigl ( ( b_0 , \dots , b_{ n } , u_{n + 1} , u_{ n + 2 } ) \notin \bigcup_{ k \le n } A_k \bigr ) \).
For each \( u_{ n + 2 } \in \R \), the set
\[
 X_{ n + 1 } ( u_{ n + 2 } ) = \setof{ ( b_0 , \dots , b_ {n} , y , u_{ n + 2 } ) }{ y \in \R } \cap A_{ n + 1 }
\] 
is finite by assumption. 
Thus, the set \( \bigcup_{q \in \Q } X_{ n + 1 } ( q ) \) is countable. 
As for the case \ref{th:main-2}\( \Rightarrow \)\ref{th:main-1}, we can restrict ourselves to the case \( \omega_1^{\Ll} = \omega_1 \). 
Therefore, there exists an \( \bar{x} \in \R \cap \Ll \) such that  \( ( b_0, \dots, b_n, \bar{x}, q ) \notin  X_{ n + 1} ( q ) \) for all \( q \in \Q \).
It follows that 
\[
\forall q \in \Q \bigl ( ( b_0 , \dots , b_ {n } , \bar{x} , q ) \notin \textstyle \bigcup_{ k \le n + 1 } A_k \bigr )
\]
and since by hypothesis the \( A_i \)s cover \( \R ^{ n + 3 } \), it follows that 
\[
\FORALL{ q \in \Q } \bigl ( ( b_0 , \dots , b_ {n} , \bar{x} , q ) \in A_{n + 2} \bigr ) ,
\]
but this means that the line determined by \( ( b_0 , \dots , b_ {n} , \bar{x} ) \) intersects \( A_{ n + 2 } \) in an infinite set, against our assumption. 
\end{proof}

There are a couple of remarks about Theorem~\ref{th:main} that we would like to make.
The first is that its proof straightforwardly relativizes to any \( a \in \R \), with \ref{th:main-1} being ``\( \constructibledegrees \) has breadth at most \( n \) above \( \eq{a}_{\mathrm{c}} \)'' and the \( A_i \)s from \ref{th:main-2} and \ref{th:main-3}  being \( \varSigma_2^1 ( a ) \).
The second  remark is that the proof of Theorem~\ref{th:main} still works if  clauses~\ref{th:main-2} and~\ref{th:main-3} are weakened a bit.
For example~\ref{th:main-2} can be weakened to: 

There are \( A_0 , \dots , A_{ n + 1 } \) covering \( \R ^{ n + 2 }  \) such that 
\begin{enumerate}
\item[(b1)]
each \( A_i \) is \( \varSigma_2^1 \), and  \( \FORALL{ \ell \in \mathcal{L}_i ( \R ^{ n + 2 } ) } ( \ell\cap A_i \text{ is thin in } \ell ) \), for all \( i \le n  \),
\item[(b2)]
\(  \FORALL{ \ell \in \mathcal{L}_{ n + 1 } ( \R ^{ n + 2 } ) } ( \ell\cap A_{ n + 1} \text{ is countable} ) \);
\end{enumerate}
and~\ref{th:main-3} can be weakened to:

There are \( A_0, \dots, A_{ n + 2 }  \) covering \( \R ^{ n + 3 }  \) such that
\begin{enumerate}
\item[(c1)]
each \( A_i \) is \( \varSigma_2^1 \), and  \( \FORALL{ \ell \in \mathcal{L}_i ( \R ^{ n + 3 } ) } ( \ell\cap A_i \text{ is thin in } \ell ) \), for all \( i \le n  \),
\item[(c2)]
\( \FORALL{ \ell \in \mathcal{L}_{n + 1 } ( \R ^{ n + 3 } ) } ( \ell\cap A_{n + 1 } \text{ is countable } ) \),
\item[(c3)]
\( \FORALL{ \ell \in \mathcal{L}_{ n + 2 } ( \R ^{ n + 3 } ) } ( \ell\cap A_{ n + 2} \text{ is not dense in } \ell ) \).
\end{enumerate}
In (b2), (c2),  and (c3) we do not require that the sets be \( \varSigma^1_2 \), a fact observed  by Törnquist and Weiss in the context of clouds---see Remark~\ref{rmk:Asger} below.
The finiteness condition in \ref{th:main-3} is weakened in  (c1) by requiring that the intersections be thin, in (c2) to be countable, and in (c3) to be non-dense---this last weakening was first brought to the fore by~\cite{Bagemihl:1961ve}.

\section{Open questions}\label{sec:conclusion}

\subsection{Covering \( \R ^2 \) with \( \varSigma_2^1 \) fogs, clouds, and sprays}\label{subsec:coveringtheplanewithsprays}

There are several results similar to the theorems by Sierpi\'{n}ski and Kuratowski asserting the equivalence between \( 2^{\aleph_0} \leq \aleph_n \) and the possibility of covering the plane with sets having small intersections with prescribed families of geometric objects.

\begin{definition}
A set \( A \subseteq \R ^2 \) is an:
\begin{itemize}
\item
\markdef{\( \aleph_k \)-fog} if for some non-zero vector \( \mathbf{u} \) called the \markdef{direction} of \( A \), each line parallel to \( \mathbf{u} \) intersects \( A \) in a set of size \( < \aleph_k \); if \( k  = 0 \), i.e. the intersections are finite, we call this a \markdef{fog}, if \( k = 1 \), i.e. the intersections are countable, we speak of \( \sigma \)-fog.
\item
 \markdef{\( \aleph_k \)-cloud} if for some point \( \mathbf{p} \in \R ^2 \), called the \markdef{center} of \( A \), each line passing through \( \mathbf{p} \) intersects \( A \) in a set of size \(  < \aleph_k \); if \( k  = 0 \) we call this a \markdef{cloud}, if \( k = 1 \) we speak of \( \sigma \)-clouds.
\item
\markdef{\( \aleph_k \)-spray} if for some point \( \mathbf{p} \in \R ^2 \), called the \markdef{center} of \( A \), each circle with  center \( \mathbf{p} \) intersects \( A \) in a set of size \(  < \aleph_k \); if \( k  = 0 \), we call this a \markdef{spray}, if \( k = 1 \)  we speak of \( \sigma \)-sprays.
\end{itemize}
\end{definition}

Fogs were introduced by R.O. Davies in~\cite{Davies:1963zr} while general \( \aleph_k \)-fogs were studied by F.~Bagemihl in~\cite{Bagemihl:1968up}, but the name ``fog'' is due to A.~Miller.

\begin{theorem}[Bagemihl, Davies]\label{th:fogs}
For every \( n \in \omega \), the following are equivalent:
\begin{enumerate}[label={\upshape (\alph*)}]
\item
\( 2^{\aleph_0} \leq \aleph_n \).
\item\label{th:fogs-1}
\( \R^2 \) can be covered with \( n + 2 \) fogs with pairwise non-parallel directions.
\item\label{th:fogs-2}
\( \R^2 \) can be covered with \( n + 1 \) \( \sigma \)-fogs with pairwise non-parallel directions.
\end{enumerate}
\end{theorem}

Each \( A_i \) in~\eqref{eq:Sierpinskigeneral} is an \( \aleph_k \)-fog with direction \( \mathbf{e}_i \), and the main theorem in~\cite{Bagemihl:1968up} generalizes the result by Kuratowski and Sierpi\'{n}ski.
Next, we consider the problem of covering the plane with clouds.

\begin{theorem}[Komjáth, Schmerl]\label{th:KomjathSchmerl}
For every \( n \in \omega \), the following are equivalent:
\begin{enumerate}[label={\upshape (\alph*)}]
\item \label{th:KomjathSchmerl-0}
\( 2^{\aleph_0} \le \aleph_n \).
\item \label{th:KomjathSchmerl-1}
\( \R ^2 \) is covered by \( n + 2 \) clouds with distinct, non-collinear centers. 
\item \label{th:KomjathSchmerl-2}
\( \R ^2 \) is covered by \( n + 1 \) \( \sigma \)-clouds with distinct, non-collinear centers. 
\end{enumerate}
\end{theorem}

The notion of cloud was introduced by  P.~Komjáth who proved in~\cite{Komjath:2001kq} the implication \ref{th:KomjathSchmerl-0}\( \IMPLIES \)\ref{th:KomjathSchmerl-1} for all \( n \), and converse implication for \( n = 1 \), while the general case of \ref{th:KomjathSchmerl-1}\( \IMPLIES \)\ref{th:KomjathSchmerl-0} is from~\cite{Schmerl:2003uq}.
Using~\cite[Theorem 2]{Erdos:1994yq} one  can easily generalize these results to \( \sigma \)-clouds, and obtain the equivalence with \ref{th:KomjathSchmerl-2}.

The plane cannot be covered with finitely many (\( \sigma \)-)clouds with collinear centers, so the non-collinearity assumption is essential.
But this is the only obstacle: for any line \( \ell \) of \( \R^2 \), \( \R ^2 \setminus \ell \) is covered by \( n + 2 \) clouds (or \( n + 1 \) \( \sigma \)-clouds) with distinct centers belonging to \( \ell \) if and only if \( 2^{\aleph_0} \le \aleph_n \). 
From now on, let \( \mathcal{P} \) denote the plane with one line removed, e.g.  \( \R^2 \) minus the \( x \)-axis
\begin{equation*}
 \mathcal{P} \coloneqq \R^2 \setminus \setof{ ( x , 0 ) }{ x \in \R } .
\end{equation*}
\begin{corollary}\label{cor:KomjathSchmerl2}
For every \( n \in \omega \), the following are equivalent:
\begin{enumerate}[label={\upshape (\alph*)}]
\item 
\( 2^{\aleph_0} \le \aleph_n \).
\item 
\( \mathcal{P} \) is covered by \( n + 2 \) clouds with distinct centers. 
\item 
\( \mathcal{P} \) is covered by \( n + 1 \) \( \sigma \)-clouds with distinct centers. 
\end{enumerate}
\end{corollary}

Next, we consider the problem of covering the plane with sprays.

\begin{theorem}[de la Vega, Schmerl]\label{th:DelaVegaSchmerl}
For every \( n \in \omega \), the following are equivalent:
\begin{enumerate}[label={\upshape (\alph*)}]
\item\label{th:DelaVegaSchmerl-0}
\( 2^{\aleph_0} \leq \aleph_n \).
\item\label{th:DelaVegaSchmerl-1}
\( \R^2 \) can be covered with \( n + 2 \) sprays with distinct, collinear centers.
\item\label{th:DelaVegaSchmerl-2}
\( \R^2 \) can be covered with \( n + 1 \) \( \sigma \)-sprays with distinct, collinear centers.
\end{enumerate}
\end{theorem}

J. H.~Schmerl introduced the notion of spray in~\cite{Schmerl:2003uq}, where he observed that \( 2^{\aleph_0} \leq \aleph_n \) implies that \( \R^2 \) can be covered with \( n + 2 \) sprays with distinct centers (and no restriction on collinearity).
R.~de la Vega in~\cite{Vega:2009xy} showed that \( \CH \) follows from the plane being covered with three sprays with collinear centers.
Moreover, in that paper, it is proved in \( \ZFC \) that the plane is the union of three sprays whose centers form an equilateral triangle, while in~\cite{Schmerl:2010nr} this was generalized to all triangles.
(This should be contrasted with Corollary~\ref{cor:KomjathSchmerl2}.)
The implication \ref{th:DelaVegaSchmerl-1}\( \IMPLIES \)\ref{th:DelaVegaSchmerl-0} was established in~\cite{Schmerl:2010nr}.
As for Theorem~\ref{th:KomjathSchmerl} the extension to \( \sigma \)-sprays follows easily from~\cite[Theorem 2]{Erdos:1994yq}.

It can be shown that if \( \R^2 \) (or \( \mathcal{P} \)) is covered with \( n + 2 \) \( \aleph_k \)-fogs (or \( \aleph_k \)-clouds, or \( \aleph_k \)-sprays) then there is a Sierpi\'{n}ski decomposition  \( A_0 , \dots , A_{n + 1} \) of \( \R^{ n + 2 } \) as in~\eqref{eq:Sierpinskigeneral}.
Moreover, if the \( \aleph_k \)-fogs/clouds/sprays belong to a pointclass \( \bGamma \supseteq \bDelta^{1}_{1} \), then the \( A_i \)s belong to \( \bGamma \) as well.

\begin{theorem}[Törnquist, Weiss]\label{th:TornquistWeiss}
The following are equivalent.
\begin{enumerate}[label={\upshape (\alph*)}]
\item\label{th:TornquistWeiss-a}
\( \R \subseteq \Ll \).
\item\label{th:TornquistWeiss-b}
\( \R^2 \) can be covered by three \( \varSigma_2^1 \) clouds with constructible, non-collinear centers.
\item
\( \mathcal{P} \) can be covered by three \( \varSigma_2^1 \) clouds with constructible, distinct centers.
\item
\( \mathcal{P} \) can be covered by two \( \varSigma_2^1 \) \( \sigma \)-clouds with constructible, distinct centers.
\end{enumerate}
\end{theorem}

\begin{remark} \label{rmk:Asger}
Törnquist and Weiss proved only the equivalence \ref{th:TornquistWeiss-a}\( \IFF \)\ref{th:TornquistWeiss-b} in~\cite{Tornquist:2015ys}, but the other two equivalences can be proved similarly.
In fact, they showed that \( \R \subseteq \Ll \) if and only if the plane can be covered  by three clouds with centers in \( \Ll \), one of which is \( \varSigma_2^1 \).
\end{remark}

In light of Theorem~\ref{th:main}, it is natural to ask:

\begin{question}\label{q:fogs-clouds-sprays}
Let \( n > 0 \).
Does ``\( \constructibledegrees \) has breadth at most \( n \)'' implies (and hence: is equivalent to) any of the following?
\begin{enumerate}[label={\upshape (\alph*)}]
\item\label{q:fogs-clouds-sprays-1}
\( \R^2 \) can be covered by \( n + 3 \) \( \varSigma_2^1 \) fogs with distinct, pairwise non-parallel directions.
\item\label{q:fogs-clouds-sprays-2}
\( \R^2 \) can be covered by \( n + 3 \) \( \varSigma_2^1 \) clouds with distinct, constructible, non-collinear  centers.
\item\label{q:fogs-clouds-sprays-3}
\( \R^2 \) can be covered by \( n + 3 \) \( \varSigma_2^1 \) sprays with distinct, constructible, collinear  centers.
\end{enumerate}
\end{question}
The list in Question~\ref{q:fogs-clouds-sprays} can be further expanded by considering the \( \sigma \)-versions of the objects, so that~\ref{q:fogs-clouds-sprays-1} becomes ``\( \R^2 \) can be covered by \( n + 2 \) \( \varSigma_2^1 \) \( \sigma \)-fogs with distinct, pairwise non-parallel directions'', and similarly for \ref{q:fogs-clouds-sprays-2} and \ref{q:fogs-clouds-sprays-3}.
Moreover for~\ref{q:fogs-clouds-sprays-2} one can also consider replacing \( \R^2 \) with \( \mathcal{P} \) and dropping the non-collinearity assumption.

%
%

\subsection{Large continuum and small breadth of constructibility degrees}\label{sec:largesacks}

It is still open whether the cardinal bound of Ditor's Theorem~\ref{th:Ditor}\ref{th:Ditor-1} is sharp also when either \( n > 2 \) and \( \alpha = 0 \) or \( n > 1 \) and \( \aleph_\alpha \) is singular. 
We refer the reader to~\cite{Wehrung:2010aa} for some positive results and more discussion. 

On top of this, we do not know whether the cardinal bound of Theorem~\ref{th:Ditor}\ref{th:Ditor-1} is optimal for the upper semi-lattice of constructibility degrees when its breadth is bigger than \( 1 \). 
As said in Section~\ref{sec:constructibility}, the iterated perfect set model witnesses the optimality of the cardinal bound for the constructibility degrees when \( n = 1 \). 
However, this is all we know. 
Hence the following question (see Remark~\ref{rmk:breadth}).

\begin{question}\label{q:breadth1}
Is it consistent relative to \( \ZF \) that \( 2^{\aleph_0} = \aleph_3 \) and that the breadth of \( \constructibledegrees \) is \( 2 \)?
\end{question}
A seemingly easier question in this direction is the following.
\begin{question}
Is \( 2^{\aleph_0} = \aleph_3 + \eqref{eq:mincont}\) consistent relative to \( \ZF \)?
\end{question}

\subsection{Covering \( \R ^3 \) with constructible continuous functions}
Let \( X \) be a set, \( n \ge 1 \) and \( f \colon X^n \to X \). 
We say that a point \( ( x_0, \dots, x_n ) \in X^{ n + 1 } \) is \emph{covered} by \( f \) if there is a permutation \( \pi \) on \( n + 1 \) such that 
\[
f \big ( x_{ \pi ( 0 ) } , \dots, x_{ \pi ( n - 1 ) }\big) = x_{ \pi ( n ) }.
\] 
A family \( \mathcal{F} \) of functions from \( X^n \) to \( X \) covers \( A \subseteq X^{ n + 1 } \) if every point of \( A \) is covered by some member of \( \mathcal{F} \).

Abraham and Geschke~\cite{Abraham:2004aa} have shown that, for each \( n \ge 2 \), it is consistent with \( \ZFC \) that \( 2^{\aleph_0} = \aleph_n \) and that \( \R ^{n} \) is covered by an \( \aleph_1 \) subset of \( C ( \R^{n - 1 } ) \).
In the iterated perfect set model (in which \( 2^{\aleph_0} = \aleph_2 \)) the following stronger property holds~\cite{Hart}, \cite[Theorem 73]{Geschke:2004aa}: \( \R ^2 \) is covered by \( C ( \R ) \cap \Ll \), where \( C ( \R ) \cap \Ll \) is the set of all continuous real functions coded in \( \Ll \).

It is easy to see that, for every \( n > 0 \), if \( \R ^{n + 1} \) is covered by \( C ( \R ^n ) \cap \Ll \), then \( \constructibledegrees \) has breadth at most \( n \): indeed, for every \( n + 1 \) reals \( ( x_0, \dots, x_n ) \in \R^{n+1}\), there would be a constructibly coded continuous function \( f \colon  \R^n \to \R \) and a permutation \( \pi \) on \( n + 1 \) such that \( x_{\pi(n)} = f ( x_{ \pi ( 0 ) } , \dots, x_{ \pi ( n - 1 ) } ) \), and thus \( x_{ \pi ( n ) } \) would be constructible relative to \( \bigoplus_{i \neq \pi ( n ) } x_i \).
Therefore, our next and last question is a more demanding version of Question~\ref{q:breadth1}, and a positive answer would yield a strengthening of Abraham and Geschke's result (at least for \( n = 3 \)).

\begin{question}\label{q:breadth2}
Is it consistent relative to \( \ZFC \) that \( 2^{\aleph_0} = \aleph_3 \) and \( \R ^3 \) is covered by \( C( \R ^2 ) \cap \Ll \)?
\end{question}

\printbibliography

\end{document}